\documentclass[12pt]{amsart}

\usepackage{amsmath}
\usepackage{amssymb}  
\usepackage{latexsym} 
\usepackage{comment}
\usepackage{url}
\usepackage{fullpage,url,amssymb,amsmath,amsthm,amsfonts,mathrsfs}
\usepackage[usenames,dvipsnames]{color}
\usepackage[pagebackref = true, colorlinks = true, linkcolor = blue, citecolor = Green]{hyperref}
\usepackage[alphabetic,lite]{amsrefs}
\usepackage{enumitem}
\usepackage{amscd}   
\usepackage[all, cmtip]{xy} 
\usepackage{xfrac}
\usepackage[T1]{fontenc}
\usepackage{subfiles}
\usepackage{colonequals}

\usepackage[all]{xy}
\usepackage{fullpage}

\usepackage{color}

\def\act#1#2%
  {\mathop{}%
   \mathopen{\vphantom{#2}}^{#1}%
   \kern-3\scriptspace%
   #2}

\newcommand{\F}{{\mathbb F}}

\newcommand{\PP}{{\mathbb P}}

\newcommand{\calC}{{\mathcal C}}

\newcommand{\calN}{{\mathcal N}}

\newcommand{\calT}{{\mathcal T}}

\newcommand{\calV}{{\mathcal V}}

\newtheorem{Theorem}{Theorem}[section]

\newtheorem{Proposition}[Theorem]{Proposition}

\newtheorem{Remark}[Theorem]{Remark}

\begin{document}

\title{Subsets of $\PP^4$ with no four points on a plane}

\author{Geertrui Van de Voorde}
\address{School of Mathematics and Statistics, University of Canterbury, Private Bag 4800, Christchurch 8140, New Zealand}
\email{geertrui.vandevoorde@canterbury.ac.nz}

\author{Jos\'e Felipe Voloch}
\address{School of Mathematics and Statistics, University of Canterbury, Private Bag 4800, Christchurch 8140, New Zealand}
\email{felipe.voloch@canterbury.ac.nz}
\urladdr{http://www.math.canterbury.ac.nz/\~{}f.voloch}

\keywords{}
\subjclass{51E22, 51E21}

\begin{abstract}
We describe a new construction of a subset of $\PP^4$ with no four points on a plane over any finite field of order $q$ in which 3 is not a square. This set has size $2q+1$, is maximal with respect to inclusion, and is the largest known such set.
\end{abstract}

\maketitle
%

\section{Introduction}

The purpose of this paper is to study subsets of $\PP^4$ over a finite field, with no four (distinct) points on a plane. These sets are known as {\em tracks} or {\em $4$-general sets}. There is an extensive literature on subsets of projective spaces with restrictions on their intersections with (linear) subspaces. Particularly, the case of no three points on a line ({\em caps}) or no $n+1$ points on a hyperplane in $\PP^n$ ({\em arcs}) have received special attention, while the intermediate cases, less so, and the case of four points on a plane in $\PP^4$ is the first such intermediate case. For a survey on the most recent results on tracks,  see \cite{zbMATH08027981}. A track is called \emph{complete} if it is not a subset of a larger track. There are similar notions for arcs and caps.

Besides their intrinsic interest, tracks are important because of their connection with error correcting codes. In \cite{deboer} (based on \cite{bose}), De Boer shows that tracks in $\mathbb{P}^N$ are equivalent to {\em almost MDS codes} (AMDS), which are $[l,k,d]$-codes with $d=l-k$.  
If the dual of an AMDS code is AMDS too, the code is {\em near MDS} (or NMDS). In that case, the corresponding track in $\mathbb{P}^N$ satisfies the additional property that every $N+2$ points are in general position.
 
It is known (see e.g. \cite{book}) that using elliptic curves, one can construct NMDS codes in $\mathbb{P}^N$ over $\F_q$, $q=p^m$, $p$ prime, of length $n$, and hence, tracks of size $n$, where

\[n=
\begin{cases}
q+\lfloor 2\sqrt{q}\rfloor & \text{if}\ p\mid \lfloor 2\sqrt{q}\rfloor \text{and}\ m\geq 3,\ m\ \text{odd}\\
q+\lfloor 2\sqrt{q}\rfloor+1 & \text{otherwise}.
\end{cases}
\]

 While the maximum length of an NMDS code of dimension $k$ is upper bounded by $2q+k$ (see \cite[Theorem 3.5]{dodunekov}), no such upper bound which is linear in $q$ is known for AMDS codes (see also Remark \ref{rem:upper}).

We denote a point in $\PP^4$ over a field $k$ by $(x_0,x_1,x_2,x_3,x_4), x_i \in k$. It is easy to see that the normal rational curve $\mathcal{N}=\{(1,t,t^2,t^3,t^4) \mid t \in k \} \cup \{(0,0,0,0,1)\}$ is a track for any field $k$; it clearly satisfies the stronger property that every five points are in general position (i.e., $\mathcal{N}$ is an arc). While over a finite field of order $q$, $q\geq 5$, it is complete as an arc (see \cite{completearc} for $q \ge 8$), this set is never complete as a track, as $(0,0,0,1,0)$ can always be added to it. As we will show, when our main result applies, there exists a complete track with $2q+1$ points properly containing it. 

Prior to our result, the largest known tracks in $\PP^4$ were the ones obtained from NMDS codes as described earlier. In \cite{giulietti}, it is shown that for large enough $q$, elliptic NMDS codes based on an elliptic curve with $j$-invariant different from $0$, are not extendable for $N>4$ and at most $2$-extendable when $N=4$. This leads to examples of tracks with at most $q+\lfloor 2\sqrt{q}\rfloor+3$ points over $\F_q$.

Our main result is as follows:

\begin{Theorem}
\label{thm:main}
If $3$ is not a square in $\F_q$, then the set
$$\{(1,t,t^2,t^3,t^4)\mid t\in \F_q\} \cup \{(0,1,2t,3t^2,4t^3) \mid t\in \F_q\} \cup \{(0,0,0,0,1)\}$$
is a complete track of size $2q+1$.
\end{Theorem}

\begin{Remark}
\label{rem:upper}
The best known upper bound for the size of a track is roughly $\sqrt{2} q^{3/2}$ so there still remains a huge gap between this bound and the lower bound $2q+1$ from Theorem \ref{thm:main}. To obtain the upper bound, one can argue as follows:
For a track $\calT$, consider the lines $\overline{xy}$ through distinct points $x,y \in \calT$. 
Then the sets $\overline{xy} \setminus \{x,y\}$ are pairwise disjoint and each have $q - 1$ points not in $\calT$, so

$$ (q - 1) {|\calT| \choose 2} \leq |\PP^4 \setminus \calT| = q^4 + q^{3} + q^2+q + 1 - |\calT|. $$

This gives the required bound.
\end{Remark}

\subsection*{Acknowledgements}
The second author was supported by the Marsden Fund administered by the Royal
Society of New Zealand. Both authors thank BIRS and the organisers of the workshop in Ramsey Theory and Finite Geometry in September 2025 for inviting us, and Jacques Verstraete for asking a question that inspired this paper as well as the argument in Remark \ref{rem:upper}. We acknowledge the use of Pari/GP, Magma, Maple and GAP for symbolic computations and ChatGPT for answering a more conceptual question, detailed in Remark \ref{rem:chatgpt}.

\section{Proof of main result}
\label{sec:main}

In what follows, we use $$\mathcal{N}=\{(1,t,t^2,t^3,t^4)\mid t\in \F_q\}  \cup  \{(0,0,0,0,1)\},$$ and $$\mathcal{V}=\{(0,1,2t,3t^2,4t^3)\mid t\in \F_q\}.$$ The hyperplane at infinity, $H_\infty$, is given by the equation $x_0=0$.  
\begin{Proposition}
\label{prop:one-point}
Assume that we are working over a field $\F_q$ of characteristic $p \ne 2,3$. 
Let $P=(0,1,2t,3t^2,4t^3)$ be a point of $\mathcal{V}$. Then $\mathcal{N}\cup \{P\}$ is a track in $\PP^4$.
\end{Proposition}

\begin{proof}
We know that $\mathcal{N}$ is a track. 
Assume, by contradiction, that there is a point $P=(0,1,2t,3t^2,4t^3)\in \mathcal{V}$ that cannot be added to $\calN$ and maintain it being a track. Then there is a plane $\pi$ through $P$ meeting $\calN$ in three points. First assume that $P_\infty=(0,0,0,0,1)$ is in $\pi$. It follows that the rank of the following matrix is $3$ for some choice of $s\neq u$:
\[\begin{bmatrix} 0&0&0&0&1\\0&1&2t&3t^2&4t^3\\1&s&s^2&s^3&s^4\\1&u&u^2&u^3&u^4 \end{bmatrix}.\]

The echelon form of this matrix is 

\[\begin{bmatrix} 1&s&s^2&s^3&s^4\\0&1&2t&3t^2&4t^3\\0&0&u^2-s^2+2t(s-u)&u^3-s^3+3t^2(s-u)&u^4-s^4+4t^3(s-u)\\0&0&0&0&1 \end{bmatrix}\]
so it has rank $3$ if and only if
\begin{align*}
(u-s)(u+s-2t)&=0\\
(u-s)(u^2+us+s^2-3t^2)&=0.
\end{align*}
Since $u\neq s$, it follows from the first equation that $t=\frac{u+s}{2}$. Substituting this value of $t$ into $3t^2=u^2+us+s^2$ yields $(u-s)^2=0$, a contradiction.

This shows that no plane though a point of $\mathcal{V}$ and $(0,0,0,0,1)$ contains $3$ points of $\mathcal{N}$. Now assume that $P_\infty=(0,0,0,0,1)$ is not in $\pi$. Let $H_a=[a_0,a_1,a_2,a_3,1]$ and $H_b=[b_0,b_1,b_2,b_3,1]$ be two distinct hyperplanes through $\pi$, and note that none of these contains $P_\infty$. The polynomial $f_a(x)=a_0+a_1x+a_2x^2+a_3x^3+x^4$ has roots exactly in the points of $\calN\setminus\{P_\infty\}$ contained in $H_a$. Since $\pi$ lies in $H_a$ and $H_b$, and contains $3$ points of $\calN\setminus\{P_\infty\}$, we have that $f_a(x)=(x-\alpha)p(x)$ and $f_b(x)=(x-\beta)p(x)$ where $p(x)$ is a monic cubic polynomial with three distinct roots, corresponding to the points of $\pi\cap \calN$, and $\alpha\neq \beta$. The point $P(0,1,2t,3t^2,4t^3)$ lies on $H_a$, which implies that $a_1+2a_2t+3a_3t^2+4t^3=0$, that is $f'_a(t)=0$. Similarly, $f'_b(t)=0$. Since $f'_a(x)=-\alpha p(x)+xp'(x)$ and $f'_b(x)=-\beta p(x)+xp'(x)$, it follows that $\alpha p(t)=tp'(t)=\beta p(t)$. Since $\alpha\neq \beta$, it follows that $p(t)=0$. This implies that $f_a(t)=0=f_b(t)$, and since $f'_a(t)=0=f'_b(t)$, we have that $t$ is a double root of $f_a$ and of $f_b$. But since $p(x)$ has three distinct roots, this implies that $\alpha=t=\beta$, a contradiction.
\end{proof}

\begin{Proposition}
\label{prop:two-points} 
Let $P\neq Q$ be two points of $\mathcal{V}$.
If $3$ is not a square in $\F_q$, then $\mathcal{N}\cup \{P,Q\}$ is a track in $\PP^4$.
\end{Proposition}

\begin{proof}
Assume, by contradiction, that there are two distinct points $P=(0,1,2t,3t^2,4t^3)$,\linebreak $Q=(0,1,2s,3s^2,4s^3)$ that cannot be added to $\calN$. Then there is a plane $\pi$ through $P$ and $Q$ meeting $\calN$ in two points. If $\pi$ contains $(0,0,0,0,1)$, then $\pi$ is contained in $H_\infty$ and it does not contain any further points of $\calN$. So we may consider two hyperplanes $H_a$ and $H_b$ through $\pi$, not containing the point $P_\infty$, which correspond to a quartic polynomial $f_a(x)$, resp. $f_b(x)$. Since $P$ and $Q$ are in $H_a\cap H_b$, we find that $f'_a(s)=f'_a(t)=0$ and $f'_b(s)=f'_b(t)=0$. The polynomials $f_a(x)$ and $f_b(x)$ have two common roots, say $\gamma,\delta$, corresponding to the intersection points of $\pi$ with $\calN$. We see that $f_a(x)=(x-\gamma)(x-\delta)r_a(x)=h(x)r_a(x)$, and $f_b(x)=(x-\gamma)(x-\delta)r_b(x)=h(x)r_b(x)$. If we consider $F(x)=r_b(t)f_a(x)-r_a(t)f_b(x)$ and $G(X)=r_b(s)f_a(x)-r_a(s)f_b(x)$, we see that $F(t)=F'(t)=0$ and $G(s)=G'(s)=0$. If we divide $F(x)$ and $G(x)$ by their leading coefficient, and denote the resulting monic polynomials by $\tilde{F}(x)$ and $\tilde{G}(x)$, we find that $$\tilde{F}(x) = h(x)(x-t)^2, \tilde{G}(x)=h(x)(x-s)^2.$$
Furthermore, we have that $\tilde{F}'(s)=0$ and $\tilde{G}'(t)=0$ and hence,

$$2h(t)(t-s)+h'(t)(t-s)^2 = 2h(s)(s-t)+h'(s)(s-t)^2 = 0.$$

Writing $h(x)=x^2+ax+b$,  it follows that 
\begin{align*}
2t^2+2at+2b+(2t+a)(t-s)&=0\\
2s^2+2as+2b+(2s+a)(s-t)&=0,
\end{align*}
and hence $a=-(s+t)$ and $b=st-\frac{(s-t)^2}{2}$. Since $h$ has $2$ roots, $\gamma$ and $\delta$, the discriminant $a^2-4b=3(s-t)^2$ is a square; hence, if $3$ is not a square, we find a contradiction.
\end{proof}

\begin{Remark}
A variant of Proposition \ref{prop:two-points} works in characteristics $2$ and $3$. In the case of characteristic $3$, the discriminant of the polynomial $h$ in the proof is $0$, so it has a double root which also leads to a contradiction. In characteristic $2$, the discriminant is not relevant but a slight variation of the argument carries through. Proposition \ref{prop:one-point} also works in characteristics $2, 3$ but Theorem \ref{thm:construction} does not because the added points are contained in the line $x_0=x_2=x_4=0$ in characteristic $2$, and in the plane $x_0=x_3=0$ in characteristic $3$. 
\end{Remark}

\begin{Theorem}
\label{thm:construction}
If $3$ is not a square in $\F_q$, then the set
$$\mathcal{N}\cup \mathcal{V}=\{(1,t,t^2,t^3,t^4)\mid t\in \F_q\} \cup \{(0,1,2t,3t^2,4t^3) \mid t\in \F_q\} \cup \{(0,0,0,0,1)\}$$
is a track of size $2q+1$.
\end{Theorem}

\begin{proof} We already know that 
$\mathcal{N}$
is a track. The set $\calV$ 
is also a track: $\mathcal{V}\cup \{(0,0,0,0,1)\}$ forms a normal rational (cubic) curve in the hyperplane $H_\infty$, and hence, no four points of $\mathcal{V}\cup \{(0,0,0,0,1)\}$ are coplanar. 
A plane cannot meet $\mathcal{N}$ in $3$ points and $\mathcal{V}$ in $1$ point by Proposition \ref{prop:one-point} and a plane cannot meet $\mathcal{N}$ in $2$ points and $\mathcal{V}$ in $2$ points Proposition \ref{prop:two-points}. A plane meeting $\mathcal{V}$ in $3$ is points is contained in the hyperplane $H_\infty$ so clearly does not meet $\mathcal{N}$.
\end{proof}

\begin{Theorem} \label{thm:complete} If $3$ is not a square in $\F_q$ and $ q \ge 89$, the track $$\mathcal{N}\cup \mathcal{V}=\{(1,t,t^2,t^3,t^4)\mid t\in \F_q\} \cup \{(0,1,2t,3t^2,4t^3) \mid t\in \F_q\} \cup \{(0,0,0,0,1)\}$$ is complete.
\end{Theorem}
\begin{proof}
In \cite{completearc}, it is shown that, for odd $q\geq 7$, a normal rational curve is complete as arc in $3$-space. 
Therefore, we know that no point of $H_\infty$ can be added. We will show that, if $q\geq 89$, every point of $\PP^4$ lies on a plane spanned by two points of $\{(0,1,2t,3t^2,4t^3) \mid t\in \F_q\} \cup \{(0,0,0,0,1)\}$ and one point of $\{(1,t,t^2,t^3,t^4)\mid t\in \F_q\} $.

To show this, we need to show that for every choice of $a,b,c,d$, there are $s,t,u$ with $s\neq t$ such that the rank of the following matrix $A$ is $3$

\[A=\begin{bmatrix} 1&a&b&c&d\\
0&1&2t&3t^2&4t^3\\
0&1&2s&3s^2&4s^3\\
1&u&u^2&u^3&u^4
\end{bmatrix},\]
or such that there are $t,u$ such that the rank of the following matrix $B$ is $3$
\[B=\begin{bmatrix} 1&a&b&c&d\\
0&1&2t&3t^2&4t^3\\
0&0&0&0&1\\
1&u&u^2&u^3&u^4
\end{bmatrix}.\]

The echelon form of $B$ is

\[\begin{bmatrix}
1 & a & b & c & d 
\\
 0 & 1 & 2 t & 3 t^{2} & 4 t^{3} 
\\
 0 & 0 & 2 a t-2 t u+u^{2}-b & 3 t^{2} a-3 t^{2} u+u^{3}-c & 4 t^{3} a-4 t^{3} u+u^{4}-d 
\\
 0 & 0 & 0 & 0 & 1 
\end{bmatrix}.\]
Therefore, the rank of $B$ is $3$ if and only if we can find $u,t$ such that 
\begin{align*}
 2 a t-2 t u+u^{2}-b&=0\\
 3 t^{2} a-3 t^{2} u+u^{3}-c&=0.
 \end{align*}
 
If $(1,a,b,c,d)=(1,a,a^2,a^3,d)$ then this is clearly the case. Otherwise, there is a solution to this system if and only if there is a $u$ which is a solution to 
\begin{equation}
\label{eq:denominator}
 u^4-4au^3+6bu^2-4cu+4ac-3b^2=0.
\end{equation}
(This follows from substituting $t=\frac{b-u^2}{2(a-u)}$ in the second equation which reads $3(a-u)t^2=c-u^3$.)

 Hence, from now on, we may assume that equation (\ref{eq:denominator}) does not have any solutions in $u$. Note that if $b=a^2$, the equation $u^4-4au^3+6bu^2-4cu+4ac-3b^2=0$ has the solution $u=a$, so we assume that $b\neq a^2$.

The echelon form of $A$ is

\[\begin{bmatrix}1 & a & b & c & d 
\\
 0 & 1 & 2 t & 3 t^{2} & 4 t^{3} 
\\
 0 & 0 & 2 s-2 t & 3 s^{2}-3 t^{2} & 4 s^{3}-4 t^{3} 
\\
 0 & 0 & 0 & A_{44}&A_{45}\end{bmatrix}\]
where
\[A_{44}= -3 a s t+3 s t u-\frac{3}{2} s \,u^{2}-\frac{3}{2} t \,u^{2}+u^{3}+\frac{3}{2} b s+\frac{3}{2} b t-c \] and
\[A_{45}=-4 a \,s^{2} t-4 t^{2} a s+4 s^{2} t u-2 s^{2} u^{2}+4 t^{2} u s-2 s t \,u^{2}-
2 t^{2} u^{2}+u^{4}+2 b \,s^{2}+2 b s t+2 b \,t^{2}-d. \]

Therefore, since $s\neq t$, the rank of $A$ is $3$ if and only if we can find $s\neq t$ and $u$ such that 
$A_{44}(s,t,u)=0$ and $A_{45}(s,t,u)=0$.

This is 
\begin{align}
g=\frac{3}{2}\left(s+t\right)\left(u^{2}-b\right)+3 s t \left(u-a\right)-c+u^{3}&=0 \label{eqg}\\
f=t^2(4s(u-a)-2(u^2-b))+t(4s^2(u-a)-2s(u^2-b))-2s^2(u^2-b)+u^4-d&=0.
\end{align}
Now consider 
\begin{align*}j&=3(u-a)f-(4(u-a)(s+t)+6(u^2-b))g=\\
&\left(4 a \,u^{3}-u^{4}-6 b \,u^{2}-4 a c+3 b^{2}+4 c u\right) \left(s+t\right)+u^{5}-3 u^{4} a+2 b \,u^{3}+2 u^{2} c-3 d u+3 d a-2 b c.\end{align*} 
Then we know that $j=0$.

 Since we assumed that $u^{4}-4 a u^{3}+6 b u^{2}+4 a c-3 b^{2}-4 c u \neq 0$, the previous equation yields $$s+t=\frac{-3 u^{4} a+u^{5}+2 b \,u^{3}+2 u^{2} c+3 d a-2 b c-3 d u}{u^{4}-4 a u^{3}+6 b u^{2}+4 a c-3 b^{2}-4 c u}.$$ Using Equation \ref{eqg}, we find that $$st=\frac{u^{6}-9 b \,u^{4}+16 c \,u^{3}-9 d \,u^{2}+9 b d-8 c^{2}}{6(u^{4}-4 a u^{3}+6 b u^{2}+4 a c-3 b^{2}-4 c u)}.$$

Since $s,t$, the solutions to the equation $X^2-(s+t)X+st=0$, need to be in $\F_q$ and distinct, we need to find $u$ such that  $D=(s+t)^2-4st$ is a non-zero square.
This is the case if and only if $9 D \left(u^{4}-4 a u^{3}+6 b u^{2}+4 a c-3 b^{2}-4 c u\right)^{2}$ is a square. The latter equals $3 F(u)$, where 

\begin{align*}F(u)&=u^{10}-10 a \,u^{9}+\left(27 a^{2}+18 b\right) u^{8}+\left(-108 a b-12 c\right) u^{7}+\left(84 a c+126 b^{2}\right) u^{6}-252 b c \,u^{5}+\\
&\quad \left(-54 a^{2} d+108 a c b-54 b^{3}+54 b d+156 c^{2}\right) u^{4}+\left(\left(108 b d-192 c^{2}\right) a+72 b^{2} c-108 c d\right) u^{3}+\\
&\quad \left(108 a c d-162 b^{2} d+72 b \,c^{2}+27 d^{2}\right) u^{2}+\left(-54 a \,d^{2}+108 b c d-64 c^{3}\right) u+\\
&\quad 27 a^{2} d^{2}+\left(-108 b c d+64 c^{3}\right) a+54 b^{3} d-36 b^{2} c^{2}.\end{align*}
 
 Hence, we need to find a point $(u,v)$ in $AG(2,q)$ on the curve $\mathcal{C}$ defined by $3F(u)=v^2$. If $F(u)$ is not a perfect square, this curve is irreducible and is a hyperelliptic curve of genus at most $[(\deg F - 1)/2] = 4$. The Hasse-Weil bound gives that the nonsingular model of $\mathcal{C}$ has at least $q+1 - 8\sqrt{q}$ points. The two points at infinity of $\mathcal{C}$ are not rational, as the leading coefficient of $3F(u)$ is not a square by hypothesis. Note that we have assumed that equation (\ref{eq:denominator}) does not have any solutions in $\F_q$, so we do not have to worry about the values of $u$ where the denominator of $D$ is zero. We need to exclude the (at most) $10$ points where $F(u)=0$. So if $q+1-8\sqrt{q}\geq 10$, i.e. if $q\geq 89$, we find a point $(u,v)$ on $\mathcal{C}$ giving rise to a solution $s\neq t\in \F_q$.

 We will now show that $F(u)$ is indeed not a perfect square.
Assume to the contrary that $F(u)=(u^5+\lambda_4 u^4+\lambda_3u^3+\lambda_2 u^2+\lambda_1 u+\lambda_0)^2$, then we find the following equations

\begin{align}
-10a& = 2\lambda_4\\
27a^2 + 18b &= \lambda_4^2 + 2\lambda_3\\
-108ab - 12c &= 2\lambda_4\lambda_3+ 2\lambda_2\\
84ac + 126b^2 &= 2\lambda_4\lambda_2 + \lambda_3^2 + 2\lambda_1\\
-252bc &= 2\lambda_4\lambda_1 + 2\lambda_3\lambda_2+ 2\lambda_0\\
-54a^2d + 108abc - 54b^3 + 54bd + 156c^2 &= 2\lambda_4\lambda_0 + 2\lambda_3\lambda_1+ \lambda_2^2 \label{eq6}\\
(108bd - 192c^2)a + 72b^2c - 108cd &= 2\lambda_3\lambda_0 + 2\lambda_2\lambda_1\label{eq7}\\
108acd - 162b^2d + 72bc^2 + 27d^2 &= 2\lambda_2\lambda_0 + \lambda_1^2\label{eq8}\\
-54ad^2 + 108bcd - 64c^3 &= 2\lambda_1\lambda_0\label{eq9}\\
27a^2d^2 + (-108bcd + 64c^3)a + 54b^3d - 36b^2c^2 &= \lambda_0^2\label{eq10}
\end{align}
The first five equations uniquely determine $\lambda_0,\ldots,\lambda_4$; we find that 
\begin{align*}
\lambda_4&=-5a\\
\lambda_3&= a^2+9b\\
\lambda_2&=5a^3-9ab-6c \\
\lambda_1&= \frac{49}{2}a^4-54ba^2+12ca+\frac{45}{2}b^2\\
\lambda_0&=\frac{1}{2}(235a^5 - 612ba^3 + 132ca^2 + 387b^2a -144cb).
\end{align*}

Equation \eqref{eq6} gives
 $$\left(-54 a^{2}+54 b\right) d+1101 a^{6}-3303 a^{4} b+696 c a^{3}+2781 a^{2} b^{2}-936 a b c-459 b^{3}+120 c^{2}=0$$ so, since $b-a^2\neq 0$, this equation uniquely determines $d$ as a function of $a,b,c$; we have
 $$d= \frac{367a^6 - 1101a^4b + 232a^3c + 927a^2b^2 - 312abc - 153b^3 + 40c^2}{18(a^2 - b)}.$$

Substituting $d$ into Equation \eqref{eq7} gives us
$-240\frac{(a^6 - 3a^4b + 4a^3c - 6abc + 3b^3 + c^2)(2a^3 - 3ab + c)}{a^2 - b}=0.$
so either $2a^3 - 3ab + c=0$ or $a^6 - 3a^4b + 4a^3c - 6abc + 3b^3 + c^2=0$.
In the former case, Equation \eqref{eq8} simplifies to $\frac{135}{2}(a^2-b)^4=0$, a contradiction since $b\neq a^2$.
Hence, we may assume that $a^6 - 3a^4b + 4a^3c - 6abc + 3b^3 + c^2=0$.
This equation has a solution in $c$ if and only if the discriminant $\bar{D}=3(a^2-b)^3$ is a square, and in that case, $c = -2a^3 + 3ab \pm \sqrt{3(a^2-b)^3}$. But substituting those values for $c$ into Equation \eqref{eq8} shows that $\frac{11247}{2}(a^2-b)^4=0$, a final contradiction.

\end{proof}
\paragraph{\bf Proof of Theorem \ref{thm:main}.}
Theorems \ref{thm:construction} and \ref{thm:complete} together show that Theorem \ref{thm:main} holds for $q \ge 89$.
The remaining cases are dealt with by computer calculations. We note that
in the proof of Theorem \ref{thm:complete}, we showed that, for all $q\geq 89$, all affine points are covered by planes spanned by two points at infinity from $\calV$ and one affine point from $\calN$. This is not true in general: a computer calculation shows that this is not true for $q=5,7,17,31$ but that the track is still complete for those values of $q$.

\begin{Remark}
\label{rem:chatgpt}
We asked ChatGPT the following question:
Let $\calN$ be the normal rational curve of degree $4$ in $4$-dimensional projective space and $L$ a line not meeting $\calN$ in the same projective space. Can you describe the set of planes containing $L$ that meet $\calN$ in at least a point and also the subset consisting of those planes that meet $\calN$ in two points?

It replied with a long description as well as the following summary: the set of planes containing $L$ is a ${\PP}^2$; those that meet $\calN$ form a plane quartic $\calC\subset{\PP}^2$ (the projection of $\calN$ from $L$); the planes that meet $\calN$ in two distinct points are precisely the finitely many nodes of $\calC$ (generically three of them).

This is correct (except in characteristic $2$, where the projection can be inseparable and the image a conic). We do not use this result in our proof but it motivated us to consider the set $\calV$ of ``derivatives'' of $\calN$ to force the curve $\calC$ to have at least two cusps (and thus, at most one node) when the line joins two points of $\calV$ as in Proposition \ref{prop:two-points}. The fact that the rationality of the remaining node depends only on the quadratic character of $3$ (and not on the choice of points of $\calV$ spanning the line) was an unexpected pleasant surprise that falls out of a calculation but for which we do not have a conceptual explanation.
\end{Remark}


\end{document}